\documentclass[12pt]{article}
\usepackage{amscd}
\usepackage{amsmath}
\usepackage{latexsym}
\usepackage{amsfonts}
\usepackage{amssymb}
\usepackage{amsthm}
\usepackage{graphicx}
\usepackage{verbatim}
\usepackage{enumerate}
\usepackage{xcolor}

\theoremstyle{plain}
\theoremstyle{definition}\newtheorem{theorem}{Theorem}[section]
\theoremstyle{plain}\newtheorem{lemma}[theorem]{Lemma}
\theoremstyle{plain}
\theoremstyle{plain}
\theoremstyle{remark}\newtheorem{remark}{Remark}[section]

\newcommand{\be}{\begin{equation}}
\newcommand{\ee}{\end{equation}}
 \newcommand{\ba}{\begin{aligned}}
 \newcommand{\ea}{\end{aligned}}

  \newcommand{\ben}{\begin{enumerate}}
   \newcommand{\een}{\end{enumerate}}

\newcommand{\Rmnum}[1]{\expandafter\@slowromancap\romannumeral #1@}

\begin{document}
%%%%%%%%%%%%%%%%%%%%%%%%%%%%%%%%%%%%%%%%%%%%%%%%%%%%%%%%%%%%%%%%%%%%%%%%%%%%%%%%%%%%%%%%%%%%%%%%%%%%

\title{ Global Regularity of
2D Generalized MHD Equations with Magnetic Diffusion}
\date{}
\maketitle
 \centerline{\scshape Quansen Jiu\footnote{The research is partially
supported by National Natural Sciences Foundation of China (No.
11171229, 11231006 and 11228102) and Project of Beijing Education
Committee.}}

{\footnotesize
  \centerline{School of Mathematical Sciences, Capital Normal University}
  \centerline{Beijing  100048, P. R. China}
   \centerline{  \it Email: jiuqs@mail.cnu.edu.cn}

\vspace{3mm}

 \centerline{\scshape Jiefeng Zhao }

\medskip
{\footnotesize
  \centerline{School of Mathematical Sciences, Capital Normal University}
  \centerline{Beijing  100048, P. R. China}
   \centerline{  \it Email: zhaojiefeng001@163.com}

%\subjclass[2000]{35Q35,76B03,76W05}
%\date{ 2013}

\begin{abstract}
\noindent

   \textbf{Abstract:} This paper is concerned with the global regularity of the  2D (two-dimensional)
   generalized  magnetohydrodynamic equations with only magnetic
diffusion $\Lambda^{2\beta} b$. It is proved that when $\beta>1 $
there exists a unique global regular solution for this equations.
The obtained result improves the previous known one which requires
that $\beta>\frac32$.
\end{abstract}

\section{Introduction}

Consider the  Cauchy problem of the following two-dimensional
generalized magnetohydrodynamic equations:
\begin{eqnarray}
\left\{
 \begin{array}{llll}\label{eq}
  u_t + u \cdot \nabla u  =  - \nabla p + b \cdot \nabla b - \nu \Lambda^{
  2\alpha} u,  \\
  b_t + u \cdot \nabla b = b \cdot \nabla u - \kappa \Lambda^{2\beta} b,\\
  \nabla \cdot u = \nabla \cdot b  =  0, \\
  u\left(x,0\right)=u_0\left(x\right),\,\,\, b\left(x,0\right)=b_0\left(x\right)
  \end{array}\right.
\end{eqnarray}
for $x\in \mathbb{R}^2$ and $t>0$, where $ u=u\left(x,t\right) $ is
the velocity, $ b =b\left(x,t\right) $ is the magnetic,
$ p =p\left(x,t\right) $ is the pressure, and $
u_0\left(x\right),\,b_0\left(x\right) $ with $\mathrm{div}
u_0\left(x\right)=\mathrm{div} b_0\left(x\right)=0$ are the initial
velocity and magnetic, respectively. Here $\nu, \kappa,
\alpha, \beta \ge 0$ are nonnegative constants and  $\Lambda$
is defined by
\begin{equation*}
  {\widehat{\Lambda f}}(\xi) = \left| \xi \right| \widehat{f}(\xi),
\end{equation*}
where $\wedge$ denotes the Fourier transform. In the following
sections, we will use the inverse Fourier transform $\vee$.

The global regularity of the 2D GMHD equations (\ref{eq}) has
attracted a lot of attention and there have been extensive studies
(see\cite{FNZ2013}-\cite{ZF2011}).
 It follows from \cite{W2011} that the
problem (\ref{eq}) has a unique global regular solution if
\begin{displaymath}
  \alpha \geqslant 1 , \ \  \beta > 0,
  \ \ \alpha + \beta \geqslant 2.
\end{displaymath}
Tran, Yu and Zhai \cite{TYZ2013} got a global regular solution under
assumptions that
\begin{equation*}
  \alpha \geqslant 1 / 2,\,\, \beta \geqslant 1 \hspace{2em}\mbox{or}\hspace{2em}
  0 \leqslant \alpha < 1 / 2,\,\, 2 \alpha + \beta > 2 \hspace{2em}\mbox{or}\hspace{2em}
  \alpha \geqslant 2, \,\,\beta = 0.
\end{equation*}
Recently, it was shown in {\cite{JZ2013}} that if $ 0\leqslant\alpha
< 1 / 2,\,\, \beta \geqslant 1, \,\,3\alpha + 2\beta >3 $, then the
solution is global regular. In particular, when $\alpha=0,
\beta>\frac 32$, the solution is global regular. This  was proved
independently in \cite{Y2013-2}\cite{YB2013}. Meanwhile, Fan,
Nakamura and Zhou {\cite{FNZ2013}} used properties of the heat
equation and presented a global regular solution when
$0<\alpha<\frac{1}{2}, \beta=1$.

In this paper, we aim at getting the global regular solution of
\eqref{eq} when $\nu=0,  \kappa>0$ and $\beta>1$. For simplicity, we
let $\kappa=1$. That is, we consider
\begin{eqnarray}
\left\{
 \begin{array}{llll}\label{eq2}
  u_t + u \cdot \nabla u  =  - \nabla p + b \cdot \nabla b,  \\
  b_t + u \cdot \nabla b = b \cdot \nabla u - \kappa \Lambda^{2\beta} b,\\
  \nabla \cdot u = \nabla \cdot b  =  0, \\
  u\left(x,0\right)=u_0\left(x\right),\,\,\,
  b\left(x,0\right)=b_0\left(x\right).
  \end{array}\right.
\end{eqnarray}

 Let $\omega = -
\partial_2 u_1 +\partial_1 u_2$ and
 $j  = - \partial_2 b_1 + \partial_1 b_2$ represent the vorticity and the current respectively.
 We will prove that $\omega, j\in L^2(0,T;L^\infty)$
 and  obtain the global regularity of the solution by the BKM type criterion in \cite{CKS1997}.
To this end, we will take advantage of the approaches  used in
{\cite{FNZ2013}} and \cite{JZ2013} to deal with the higher
regularity estimates of $j$. More precisely, using the equation
satisfied by the current $j$, we will obtain the estimates of
$\|\Lambda^r j\|_{L^2}^2 \left( t \right) + \int_0^t \|
\Lambda^{\beta+r} j\|_{L^2}^2\le C$ with $r=\beta-1$.  Using the
singular integral representation of $\Lambda^\delta j$ with some
$\delta>0$, we will obtain the estimate $\|\nabla
j\|_{L^2(0,T;L^\infty(\mathbb{R}^2))}$. Then we get the estimates of
$\|\omega\|_{L^2(0,T;L^\infty(\mathbb{R}^2))}$ using the particle
trajectory method. It should be noted that after the paper is
finished, at the almost same time, Cao, Wu and Yuan obtain the
similar result independently using a different method
(see\cite{CWY-2013}). In comparison with result obtained in
\cite{CWY-2013}, it is not required that $\|\nabla
j_0\|_{L^\infty}<\infty$ in our result.

The main result  of this paper is stated as follows.
\begin{theorem}
   Let $\beta>1$ and assume that $\left( u_0, b_0 \right)\in H^\rho$ with $\rho >\max\{2,\beta\}$. Then for any $T>0$, the Cauchy problem \eqref{eq2} has a unique   regular solution
\begin{equation*}
(u,b)\in L^\infty([0,T];H^\rho(\mathbb R^2)) \,\,\mbox{and}\,\, b\in L^2([0,T];H^{\rho+\beta}(\mathbb R^2)).
\end{equation*}
\end{theorem}
\begin{remark}
When $\alpha=0, \beta>\frac{3}{2}, \rho>2$, the result has been
obtained in \cite{JZ2013}, \cite{Y2013-2}and \cite{YB2013}.
\end{remark}

\section{Preliminaries}\label{sec:the preparations}
Let us first consider the following equation
\begin{eqnarray*}
 \left\{
 \begin{array}{llll}
   v_t +\Lambda^{2\beta}v =  f
  \\
   v\left(x,0\right)=v_0(x).
 \end{array}\right.
\end{eqnarray*}
Similar to the heat equation, we can get
\begin{equation}\label{heq}
v\left(x,t\right)=\int_{\mathbb{R}^2} t^{-\frac{1}{\beta}}h\left(\frac{x-y}{t^\frac{1}{2\beta}}\right)v_0(y) \mathrm dy+
\int_0^t \int_{\mathbb{R}^2} (t-s)^{-\frac{1}{\beta}}h\left(\frac{x-y}{(t-s)^\frac{1}{2\beta}}\right)f\left(y,s\right) \mathrm dy\mathrm d s,
\end{equation}
where $h(x)=\left(e^{-|\cdot|^{2\beta}}\right)^\vee(x)$ and it has the similar properties as the heat kernel.
\begin{lemma}Let $l$ be a nonnegative integer and $\eta\geqslant0$, then
\begin{eqnarray}
\left\|\nabla^{l} h\right\|_{L^1}+\left\|\Lambda^{\eta}h\right\|_{L^1}\leqslant C.
\end{eqnarray}
\end{lemma}
\begin{proof}
First, we  give the proof of the estimates of $\nabla^{l}h$.
\begin{eqnarray*}
\left\|\nabla^{l} h\right\|_{L^1}&=&C\sup_{|\gamma|=l}\int_{\mathbb R^2}\left|\int_{\mathbb R^2}\xi^\gamma e^{-\left|\xi\right|^{2\beta}}e^{ix\cdot \xi}\mathrm d\xi\right|\mathrm dx\\
&=&C\sup_{|\gamma|=l} \int_{\left|x\right|\leqslant 1}\left|\int_{\mathbb R^2}\xi^\gamma e^{-\left|\xi\right|^{2\beta}}e^{ix\cdot \xi}\mathrm d\xi\right|\mathrm dx+
C\int_{\left|x\right|\geqslant 1}\left|\int_{\mathbb R^2}\xi^\gamma e^{-\left|\xi\right|^{2\beta}}e^{ix\cdot \xi}\mathrm d\xi\right|\mathrm dx\\
&\leqslant & C+ C\sup_{|\gamma|=l}\int_{\left|x\right|\geqslant 1}(1+\left|x\right|^2)^{-2}\left|\int_{\mathbb R^2}\xi^\gamma e^{-\left|\xi\right|^{2\beta}}(1-\Delta_\xi)^2e^{ix\cdot \xi}\mathrm d\xi\right|\mathrm dx\\
&\leqslant & C+ C\sup_{|\gamma|=l}\int_{\left|x\right|\geqslant 1}(1+\left|x\right|^2)^{-2}\left|\int_{\mathbb R^2}(1-\Delta_\xi)^2(\xi^\gamma e^{-\left|\xi\right|^{2\beta}})e^{ix\cdot \xi}\mathrm d\xi\right|\mathrm dx\\
&\leqslant & C.
\end{eqnarray*}
Next, we start to estimate $\Lambda^\eta h$ and let $l>\eta$.
\begin{eqnarray*}
\left\|\Lambda^{\eta} h\right\|_{L^1}&=&\left\|\sum_{k\geqslant-1}\Delta_{k}\Lambda^{\eta} h\right\|_{L^1}\\
&\leqslant&\left\|\Delta_{-1}\Lambda^{\eta} h\right\|_{L^1}+\sum_{k\geqslant 0}\left\|\Delta_{k}\Lambda^{\eta} h\right\|_{L^1}\\
&\leqslant&C\left\| h\right\|_{L^1}+ C\sum_{k\geqslant 0}2^{k(-l+\eta)}\left\|\Delta_{k}\nabla^{l} h\right\|_{L^1}\\
&\leqslant&C+C\sum_{k\geqslant 0}2^{k(-l+\eta)}\left\|\nabla^{l} h\right\|_{L^1}\\
&\leqslant&C,
\end{eqnarray*}
where we use the nonhomogeneous Littlewood-Paley decompositions $Id=\sum_k \Delta_k$ and Bernstein-Type inequalities (see\cite{BCD2011}).

\end{proof}

Now, let $\omega = \nabla^{\bot} \cdot u = - \partial_2 u_1 + \partial_1 u_2$
  and $j = \nabla^{\bot} \cdot b = - \partial_2 b_1 + \partial_1 b_2$,
and applying $\nabla^{\bot} \cdot$ to the equations (\ref{eq}), we obtain the following equations for
   $\omega$ and $j$:
  \begin{eqnarray}
    \omega_t + u \cdot \nabla \omega & = & b \cdot \nabla j,  \label{eq:omega-L2}\\
    j_t + u \cdot \nabla j & = & b \cdot \nabla \omega + T \left( \nabla u,
    \nabla b \right) - \Lambda^{2 \beta} j,  \label{eq:j-L2}
  \end{eqnarray}
 where
  \begin{equation*}
    T \left( \nabla u, \nabla b \right) = {\color{black} 2 \partial_1 b_1
    \left( \partial_1 u_2 + \partial_2 u_1 \right) + 2 \partial_2 u_2  \left(
    \partial_1 b_2 + \partial_2 b_1 \right)} .
  \end{equation*}
The estimates for $\omega, j$ are obtained in \cite{TYZ2013} and \cite{JZ2013}, which is presented
in the following lemma.
\begin{lemma}
  Assume that  $\alpha = 0,\,\,\beta > 1,\, r=\beta-1,\,\,\mbox{and}\,\, k \geqslant \beta$.
  Let $u_0, b_0 \in H^k$. For any $T > 0$, we have
  \begin{eqnarray}
    \left\| \omega \right\|_{L^2}^2 \left( t \right) + \left\| j
    \right\|_{L^2}^2 \left( t \right) + \int_0^t  \left\| \Lambda^{\beta} j
    \right\|_{L^2}^2  \mathrm d\tau \leqslant C \left( T
    \right), \\
    \left\|\Lambda^r j\right\|_{L^2}^2 \left( t \right) + \int_0^t \left\| \Lambda^{\beta+r} j
    \right\|_{L^2}^2  \mathrm d\tau \leqslant C \left(T
    \right).
  \end{eqnarray}
\end{lemma}

\section{The Proof of Theorem 1.1}\label{sec:the proof of theorem 1.1}
In this section, we will prove our main result Theorem 1.1. The
proof is divided into three steps.

Step 1: $\omega\in L^\infty(0,T;L^p(\mathbb R^2)), j\in L^p(0,T;\mathbb R^2)$ for any $2<p<\infty$.

First,  the second equation in (\ref{eq}) can be rewritten as
\begin{displaymath}
b_t+\Lambda^{2\beta}b=\sum_{i=1}^2\partial_i (  b_i u-u_i b)
\end{displaymath}
Due to (\ref{heq}), we have
\begin{equation}
b\left(x,t\right)=\int_{\mathbb{R}^2} t^{-\frac{1}{\beta}}h\left(\frac{x-y}{t^\frac{1}{2\beta}}\right)b_0(y) \mathrm dy+
\int_0^t \int_{\mathbb{R}^2} (t-s)^{-\frac{1}{\beta}}h\left(\frac{x-y}{(t-s)^\frac{1}{2\beta}}\right)\sum_{i=1}^2\partial_i (  b_i u-u_i b)\left(y,s\right) \mathrm dy\mathrm d s.
\end{equation}
It follows from Lemma 2.2 that $b\in L^\infty(0,T;L^\infty(\mathbb
R^2))$ and $u\in L^\infty(0,T;L^p(\mathbb R^2))$ for any $2\leqslant
p<\infty$. Thanks to Lemma 2.1, we can get
\begin{eqnarray}
\|\nabla b\|_{L^p(0,T;\mathbb R^2)}& \leqslant & C(T)\|\nabla b_0\|_{L^p(\mathbb R^2)}+ C\|bu\|_{L^p(0,T;\mathbb R^2)} \int_0^T \left\|t^{-\frac{2}{\beta}}(\nabla^2 h)\left(\frac{\cdot}{t^\frac{1}{2\beta}}\right)\right\|_{L^1(\mathbb R^2)}\mathrm dt
\nonumber
\\
&\leqslant& C(T)\nonumber\\
\|\nabla^2 b\|_{L^2(0,T;L^p(\mathbb R^2))}& \leqslant & C\|\nabla b_0\|_{L^p(\mathbb R^2)}\left(\int_0^T \left\|t^{-\frac{3}{2\beta}}(\nabla h)\left(\frac{\cdot}{t^\frac{1}{2\beta}}\right)\right\|_{L^1(\mathbb R^2)}^2\mathrm dt\right)^{\frac{1}{2}}
\nonumber\\
&+& C\|b\cdot\nabla u-u\cdot\nabla b\|_{L^2(0,T;L^p(\mathbb R^2))} \int_0^T \left\|t^{-\frac{2}{\beta}}(\nabla^2 h)\left(\frac{\cdot}{t^\frac{1}{2\beta}}\right)\right\|_{L^1(\mathbb R^2)}\mathrm dt
\nonumber\\
&\leqslant&C(T)\|\nabla b_0\|_{L^p(\mathbb R^2)} +C(T)\|b\cdot\nabla u-u\cdot\nabla b\|_{L^2(0,T;L^p(\mathbb R^2))}\label{in1}.
\end{eqnarray}
for any $2\le p<\infty$.

Multiplying (\ref{eq:omega-L2}) by $|\omega|^{p-2}\omega(p>2)$, and
integrating with respect to $x$, we get
\begin{eqnarray*}
  \frac{1}{p}\frac{\mathrm d}{\mathrm dt} \left\| \omega \right\|_{L^p}^p
  & \leqslant &  \int_{\mathbb{R}^2} \left| b \right|  \left| \nabla j \right|
  \left| \omega \right|^{p - 1} \mathrm dx,\\
  & \leqslant & \left\| b \right\|_{L^{\infty}}\left\| \nabla j \right\|_{L^p}\left\| \omega \right\|_{L^p}^{p-1}
\end{eqnarray*}
Thus, we have
\begin{eqnarray*}
  \frac{1}{2}\frac{\mathrm d}{\mathrm dt} \left\| \omega \right\|_{L^p}^2
  & \leqslant & \left\| b \right\|_{L^{\infty}}\left\| \nabla j \right\|_{L^p}\left\| \omega \right\|_{L^p}
\end{eqnarray*}and
\begin{eqnarray*}
\left\| \omega \right\|_{L^p}^2 &\leqslant& C\left\| \omega(x,0) \right\|_{L^p}^2+ C\int_0^t (\left\| \nabla j \right\|_{L^p}^2+\left\| \omega \right\|_{L^p}^2)\mathrm ds\\
& \leqslant & C + C\int_0^t (\left\| \nabla^2 b \right\|_{L^p}^2+\left\| \omega \right\|_{L^p}^2)\mathrm ds\\
& \stackrel{(\ref{in1})}{\leqslant} & C + C\int_0^t (\left\| b\cdot\nabla u-u\cdot\nabla b \right\|_{L^p}^2+\left\| \omega \right\|_{L^p}^2)\mathrm ds\\
& \leqslant & C + C\int_0^t (\left\| \nabla b \right\|_{L^p}^2 \left\| u \right\|_{L^\infty}^2+\left\| \omega \right\|_{L^p}^2)\mathrm ds\\
& \leqslant & C + C\int_0^t (1+ \left\| \nabla b \right\|_{L^p}^2 )\left\| \omega \right\|_{L^p}^2\mathrm ds.
\end{eqnarray*}
This, combining with the Gronwall's inequality, leads to $\omega\in
L^\infty(0,T;L^p(\mathbb R^2))$ for any $2<p<\infty$.

Step 2: $\nabla j\in L^2(0,T;L^\infty(\mathbb R^2))$.

Similar to  \cite{JZ2013}, we apply
$\Lambda^\delta(0<\delta<\min\{2\beta-2,\rho-2\})$ on both sides of
\eqref{eq:j-L2} to obtain
\begin{eqnarray}
    (\Lambda^\delta j)_t + \Lambda^{2 \beta}\Lambda^\delta j =-\Lambda^\delta(u \cdot \nabla j)  +\Lambda^\delta(b \cdot \nabla \omega) + \Lambda^\delta(T \left( \nabla u,\nabla b \right)).
\end{eqnarray}
Thanks to Lemma 2.2 and Step 1, we have that $uj, b\omega,
\mbox{and}\,\, T \left( \nabla u,\nabla b \right) \in
L^p(0,T;\mathbb R^2)$ for any $2<p<\infty$. In the same way as in
Step 1, we have
\begin{eqnarray*}
\Lambda^\delta j\left(x,t\right)&=&\int_{\mathbb{R}^2} t^{-\frac{1}{\beta}}h\left(\frac{x-y}{t^\frac{1}{2\beta}}\right)\Lambda^\delta j_0(y) \mathrm dy\\&+&
\int_0^t \int_{\mathbb{R}^2} (t-s)^{-\frac{1}{\beta}}h\left(\frac{x-y}{(t-s)^\frac{1}{2\beta}}\right)(-\Lambda^\delta(u \cdot \nabla j)  +\Lambda^\delta(b \cdot \nabla \omega) ) \left(y,s\right) \mathrm dy\mathrm d s\\
&+ &\int_0^t \int_{\mathbb{R}^2} (t-s)^{-\frac{1}{\beta}}h\left(\frac{x-y}{(t-s)^\frac{1}{2\beta}}\right)\Lambda^\delta(T \left( \nabla u,\nabla b \right))\left(y,s\right) \mathrm dy\mathrm d s
\end{eqnarray*}
and
\begin{eqnarray*}
\|\nabla \Lambda^\delta j\|_{L^2(0,T;L^p(\mathbb R^2))}& \leqslant & C\|\Lambda^\delta j_0\|_{L^p(\mathbb R^2)}
\left(\int_0^T \left\|t^{-\frac{3}{2\beta}}(\nabla h)\left(\frac{\cdot}{t^\frac{1}{2\beta}}\right)\right\|_{L^1(\mathbb R^2)}^2\mathrm dt\right)^{\frac{1}{2}}
\\
&+&
C\|uj\|_{L^2(0,T;L^p(\mathbb R^2))} \int_0^T \left\|t^{-\frac{4+\delta}{2\beta}}(\Lambda^{\delta}\nabla^2 h)\left(\frac{\cdot}{t^\frac{1}{2\beta}}\right)\right\|_{L^1(\mathbb R^2)}\mathrm dt
\nonumber
\\
&+&
C\|bw\|_{L^2(0,T;l^p(\mathbb R^2))} \int_0^T \left\|t^{-\frac{4+\delta}{2\beta}}(\Lambda^{\delta}\nabla^2 h)\left(\frac{\cdot}{t^\frac{1}{2\beta}}\right)\right\|_{L^1(\mathbb R^2)}\mathrm dt
\nonumber
\\
&+&
C\|T \left( \nabla u,\nabla b \right)\|_{L^2(0,T;L^p(\mathbb R^2))} \int_0^T \left\|t^{-\frac{3+\delta}{2\beta}}(\Lambda^{\delta}\nabla h)\left(\frac{\cdot}{t^\frac{1}{2\beta}}\right)\right\|_{L^1(\mathbb R^2)}\mathrm dt
\nonumber
\\
&\leqslant& C(T)
\end{eqnarray*}
for any $2<p<\infty$. So we can choose $\delta$  small and $p$ large
enough such that $\nabla j\in L^2(0,T;L^\infty(\mathbb R^2))$ and
$\|\Lambda^\delta j_0\|_{L^p}\leqslant C \| j_0\|_{H^\rho}$.

Step 3: $\omega \in L^{\infty} \left( 0, T ; L^{\infty} \right)$.

Because of the estimates of the step 2, and the following equation
\begin{equation*}
    \omega_t + u \cdot \nabla \omega = b \cdot \nabla j,
\end{equation*}
we can prove that $\omega \in L^{\infty} \left( 0, T ; L^{\infty}
\right)$ by using the particle trajectory method. By taking
advantage of the BKM type criterion for global regularity (see
\cite{CKS1997}), we finish the proof of Theorem 1.1.

\end{document}